\documentclass[11pt,reqno]{amsart}
\usepackage[T1]{fontenc}
\usepackage{lmodern}
\usepackage{amsmath,amssymb,mathtools}
\usepackage{microtype}
\usepackage{graphicx}
\usepackage{enumitem}
\usepackage[a4paper,margin=28mm]{geometry}
\usepackage{xcolor}
\definecolor{linkblue}{RGB}{0,0,255}
\usepackage[colorlinks=true,linkcolor=linkblue,citecolor=linkblue,urlcolor=linkblue]{hyperref}
\numberwithin{equation}{section}
\newtheorem{theorem}{Theorem}[section]
\newtheorem{lemma}[theorem]{Lemma}
\newtheorem{proposition}[theorem]{Proposition}
\newtheorem{corollary}[theorem]{Corollary}
\theoremstyle{definition}
\newtheorem{definition}[theorem]{Definition}
\theoremstyle{remark}
\newtheorem{remark}[theorem]{Remark}
\newcommand{\R}{\mathbb R}
\newcommand{\E}{\mathcal E}
\newcommand{\Tmax}{T_{\max}}
\newcommand{\norm}[1]{\lVert #1\rVert}
\newcommand{\pair}[2]{\langle #1,#2\rangle}
\DeclareMathOperator*{\essinf}{ess\,inf}
\DeclareMathOperator*{\esssup}{ess\,sup}
\setlist[enumerate]{leftmargin=*,label=\textup{(\roman*)}}
\allowdisplaybreaks[1]

\title[Blow-up for a wave equation with a variable logarithmic source]{Blow-up at arbitrary
energy levels for a strongly damped viscoelastic wave equation
with a variable-exponent logarithmic source}
\author[M.  Liao]{Menglan Liao}
\address{School of Mathematics, Hohai University, Nanjing, 210098, China \&
Laboratory of Mathematical Modeling and Intelligent Computing for Water Systems, Hohai University, Nanjing, 211100, China}
 \email{liaoml@hhu.edu.cn}
 
\author[M. Zhang]{Mingxue Zhang}
\address{Institute for Math and AI, Wuhan University, Wuhan, 430072, China. }
 \email{zhangmingxuex@163.com}
 
\thanks{This work was supported by the National Natural Science Foundation of China (12401290) and the Natural Science Foundation of Jiangsu Province (BK20230946). }

\date{}

\begin{document}
\begin{abstract}
We study a strongly damped viscoelastic wave equation
with a logarithmic source whose measurable exponent
satisfies $2<p_-\le p(x) \le p_+<2^*$.
An optimal pointwise correction to the source--potential
inequality defines a \emph{shifted energy}.
Under suitable kernel conditions, we prove finite-time
blow-up for negative shifted energy and for a class of
nonnegative shifted energies with no fixed upper energy
threshold.
Explicit upper lifespan bounds follow from an estimate
for reaching negative shifted energy and a concavity
argument.
We also construct smooth blow-up data at every prescribed
real energy level and show that our lifespan estimate
covers additional data in the constant-exponent case.
Explicit examples are given to illustrate the blow-up results.
\end{abstract}
\maketitle

\section{Introduction}

Let $\Omega\subset\R^n$ with $n\ge 1$ be a bounded domain with smooth boundary.
This paper is concerned with the initial-boundary value problem for a strongly damped viscoelastic wave equation involving variable-exponent logarithmic source terms:
\begin{equation}\label{problem}
\begin{cases}
u_{tt}-\Delta u
+\displaystyle\int_0^t g(t-s)\Delta u(s)\,ds
-\Delta u_t
=|u|^{p(x)-2}u\log|u|,
&(x,t)\in\Omega\times(0,\infty),\\
u=0,
&(x,t)\in\partial\Omega\times(0,\infty),\\
u(0)=u_0,\quad u_t(0)=u_1,
&x\in\Omega.
\end{cases}
\end{equation}
The source is defined to be zero at $u=0$.
The convolution describes memory, and $-\Delta u_t$
provides strong damping. The initial values $(u_0,u_1)\in H_0^1(\Omega)\times L^2(\Omega)$ are prescribed.
The goal of this paper is to obtain
blow-up criteria and explicit upper lifespan bounds
for problem \eqref{problem}, including initial data with
arbitrarily high energy.

The global existence and blow-up of damped wave equations
have been studied extensively.
Gazzola and Squassina~\cite{GazzolaSquassina2006}
investigated semilinear wave equations with weak and strong
damping, established global existence and blow-up results,
and constructed high-energy initial data leading to
finite-time blow-up.
For viscoelastic equations, Messaoudi~\cite{Messaoudi2006}
proved finite-time blow-up for a class of positive-energy
initial data in the presence of nonlinear velocity damping
and a \emph{power source}.
Song and Xue~\cite{SongXue2014} considered a viscoelastic
wave equation with strong damping and a power source
and obtained blow-up results for suitable initial data
with arbitrarily high energy.
The local theory for related viscoelastic equations
with supercritical sources and damping was developed by
Guo et al.~\cite{GuoEtAl2014}, who established
existence, uniqueness, and continuous dependence on the
initial data. Subsequently, Guo, Rammaha, and Sakuntasathien~\cite{GuoRammahaSakuntasathien2017}
established blow-up results for negative initial
energy and for suitable positive-energy data with
sufficiently large initial quadratic energy,
under conditions ensuring that the source dominates
the dissipative effects. These results highlight the interplay between the source,
damping, and memory in determining global existence
and finite-time blow-up, and motivate the study of
more general source terms.

For equations with \emph{logarithmic sources}
and no memory term, Ma and Fang~\cite{MaFang2018}
studied a strongly damped semilinear wave equation
under homogeneous Dirichlet boundary conditions.
Using a family of potential wells, the logarithmic
Sobolev inequality, and a perturbed energy method,
they established sufficient conditions for global
existence and infinite-time blow-up, together with
energy decay estimates.
Related results for strongly damped semilinear
wave equations with logarithmic sources were
obtained in~\cite{DiShangSong,ZuGuo}. For $p$-Laplacian wave equations with weak and
strong damping and logarithmic sources, we refer
to \cite{WuYangCheng2023,YangHan2022} and the
references therein.   For viscoelastic equations with \emph{constant-exponent logarithmic sources}, Ha and Park~\cite{HaPark}
proved local existence and blow-up results for
problem \eqref{problem}. They pointed out that blow-up occurs if the initial energy is controlled below a fixed potential-well depth.
Subsequently, Liao~\cite{Liao} improved this result and obtained a high-energy blow-up criterion
and lifespan estimates.
In the notation of that paper, the positive-energy
criterion reads
\begin{equation}\label{2}
0<E(0)<\frac Cp(u_0,u_1),
\end{equation}
where $(\cdot,\cdot)$ is the $L^2(\Omega)$ inner product
and $C>0$ depends on the exponent, the domain, and the
kernel.
This criterion compares energy with the initial
displacement--velocity pairing rather than with a fixed
potential-well depth.
The upper estimate in \cite[Theorem~3.2]{Liao}
additionally assumes
\begin{equation}\label{3}
E(0)\le\frac{C}{2p}\norm{u_0}_2^2.
\end{equation}

Variable-exponent nonlinearities allow the growth of a nonlinear response to vary with position and thus provide a natural framework for studying spatially heterogeneous problems. An important motivation comes from electrorheological fluids, whose constitutive laws lead to equations with nonstandard growth conditions, see for instance \cite{Ruzicka}. Variable-exponent models have also found important applications in image processing. In particular, Chen, Levine, and Rao \cite{ChenLevineRao2006} proposed and analyzed a variable-exponent functional for image denoising, enhancement, and restoration.  
For evolution equations with nonstandard growth, including
local/global existence, uniqueness, and blow-up, we refer
to \cite{AntontsevShmarev}.
 Le, Le, and Nguyen~\cite{LeLeNguyen2023} studied
viscoelastic equations with strong damping and
\emph{variable-exponent sources}, using the potential well
method to establish decay of global solutions and
finite-time blow-up. High-energy blow-up has also been studied for equations
with variable exponents. For instance, 
Talahmeh, Messaoudi and Alahyane
\cite{TalahmehMessaoudiAlahyane}
considered a viscoelastic equation with a variable power
source  and linear velocity damping.
They proved finite-time blow-up for suitable initial data
with arbitrarily large positive energy under conditions
on the relaxation function $g$.
For blow-up results for equations involving the
variable-exponent damping $|u_t|^{m(x)-2}u_t$
and source $|u|^{p(x)-2}u$, we refer to
\cite{DingZhou2023,LiaoGuoZhu,LiaoTan,MessaoudiTalahmehAlSmail,ParkKang2019}
and references therein. These studies motivate the investigation of
\emph{variable-exponent logarithmic sources}, particularly
in the presence of strong damping and memory.
A natural question is whether the high-energy blow-up
criterion and lifespan estimates established in our
previous work~\cite{Liao} can be extended to this setting. However, this extension is not straightforward, because the spatial variation of the exponent
requires a correction to the source--potential
inequality used in the constant-exponent case.
The present paper addresses this question and also
seeks to obtain an upper lifespan bound under the
high-energy criterion itself, without the additional
initial-data restriction imposed in~\cite{Liao}.


The first issue is the source--potential inequality.
For a fixed exponent $p$, the potential
\[
F_p(r)=\frac{|r|^p}{p}\log|r|-\frac{|r|^p}{p^2}
\]
satisfies
$|r|^p\log|r|-pF_p(r)=|r|^p/p\ge0$.
In the variable-exponent case, set
\[
F(x,r)=\frac{|r|^{p(x)}}{p(x)}\log|r|
-\frac{|r|^{p(x)}}{p(x)^2},
\]
we have
\[
|r|^{p(x)}\log|r|-p_-F(x,r)
=
|r|^{p(x)}
\left[
\frac{p(x)-p_-}{p(x)}\log|r|
+\frac{p_-}{p(x)^2}
\right],
\]
which can be negative for small $|r|$. For almost every fixed $x\in\Omega$, we determine
the smallest nonnegative constant $c_{p_-}(p(x))$,
independent of $r$, such that
\[
|r|^{p(x)}\log|r|-p_-F(x,r)
+c_{p_-}(p(x))\ge0
\quad\text{for all }r\in\mathbb R,
\]
where $|r|^{p(x)}\log|r|$ is understood to be zero
at $r=0$. $C_*=\int_\Omega c_{p_-}(p(x))\,dx$ leads to the \emph{shifted energy}
$\E=E+C_*/p_-$.
The shift preserves the dissipation identity and
restores the inequality needed for the growth and
concavity estimates.

The second issue is quantitative.
The growth argument in \cite[Lemma~2.7]{Liao}
can be adapted to the shifted energy $\E$, but an explicit lifespan
bound still requires control of the time before
the shifted energy becomes negative.
We prove that the solution either has already ceased
to exist by an explicit time $\tau$, or reaches
a fixed negative shifted-energy level before $\tau$.
Up to that level, dissipation controls both the
displacement and its gradient.
We then apply the lifespan estimate for negative
shifted energy to bound the remaining lifespan.
This gives a time bound under the positive
shifted-energy criterion without an additional
condition of the form \eqref{3}.


The main contributions are thus the explicit correction
for the variable-exponent logarithmic source and a quantitative
passage from nonnegative to negative shifted energy.
The classical concavity criterion
provides the supporting framework.
To make the comparison with \cite{Liao} concrete,
we construct constant-exponent data satisfying
\eqref{2} but violating \eqref{3}.
We also construct smooth blow-up data at every
prescribed real energy level. We provide explicit initial data satisfying the three
shifted-energy blow-up criteria for a nonconstant
exponent and an exponential memory kernel.

The paper is organized as follows.
Section~\ref{main} states the assumptions and main results.
Section~\ref{estimates} shows the corrected source-potential inequality and lifespan estimates.
Section~\ref{proofs} proves the blow-up theorems and
constructs the comparison data. Section~\ref{numerical} presents explicit examples
 illustrating the blow-up criteria.

\section{Assumptions and main results}\label{main}
Let us write
\[
V=H_0^1(\Omega),\quad H=L^2(\Omega),\quad \norm{z}_V=\norm{\nabla z}_2.
\]
The dual pairing between $V^*$ and $V$ is denoted by $\pair{\cdot}{\cdot}$. If $\lambda_1>0$ is the first Dirichlet eigenvalue, then
\begin{equation}\label{poincare}
\lambda_1\norm{z}_2^2\le\norm{\nabla z}_2^2\quad \text{for all }z\in V.
\end{equation}
We impose the following assumptions.
\begin{enumerate}[label=\textup{(H\arabic*)}]
\item The exponent $p:\Omega\to\R$ is measurable and
\begin{equation*}
2<p_-:=\essinf_{x\in\Omega}p(x)\le p(x)\le p_+:=\esssup_{x\in\Omega}p(x)<2^*
\end{equation*}
almost everywhere, where $2^*=2n/(n-2)$ if $n\ge3$, and $2^*=\infty$ if $n\le2$.
\item The kernel satisfies
\begin{equation*}
g\in C^1([0,\infty)),\quad g\ge0,\quad g'\le0,\quad
\ell:=1-\int_0^\infty g(s)\,ds>0.
\end{equation*}
\end{enumerate}
Note that no continuity assumption on $p$ is needed, since the principal spatial operators have fixed growth and the source is estimated via fixed-exponent Sobolev embeddings.

\begin{definition}[Energy solution]\label{solution}
For $(u_0,u_1)\in V\times H$, an energy solution on $[0,T]$ satisfies
\begin{equation}\label{regularity}
\begin{gathered}
u\in C([0,T];V),\quad u_t\in C([0,T];H)\cap L^2(0,T;V),\quad 
u_{tt}\in L^2(0,T;V^*)
\end{gathered}
\end{equation}
with $u(0)=u_0,~u_t(0)=u_1,$ and, for every $z\in V$ and almost every $t\in(0,T)$,
\begin{equation*}
\pair{u_{tt}}z+(\nabla u,\nabla z)+(\nabla u_t,\nabla z)
-\int_0^t g(t-s)(\nabla u(s),\nabla z)\,ds
=\int_\Omega |u|^{p(x)-2}u\log|u|z\,dx.
\end{equation*}
\end{definition}

Define
\begin{equation*}
F(x,r)=\frac{|r|^{p(x)}}{p(x)}\log|r|-\frac{|r|^{p(x)}}{p(x)^2},\quad F(x,0)=0,
\end{equation*}
and
\begin{equation*}
(g\circ\nabla u)(t)=\int_0^t g(t-s)\norm{\nabla u(t)-\nabla u(s)}_2^2\,ds.
\end{equation*}
The energy is
\begin{equation}\label{energy}
\begin{split}
E(t):=\frac12\norm{u_t(t)}_2^2
+\frac12\left(1-\int_0^t g(s)\,ds\right)\norm{\nabla u(t)}_2^2+\frac12(g\circ\nabla u)(t)-\int_\Omega F(x,u(t))\,dx.
\end{split}
\end{equation}
For $q\ge p_-$, set
\begin{equation}\label{correction}
c_{p_-}(q):=
\begin{cases}
\displaystyle\frac{q-p_-}{q^2}e^{-\frac{q}{q-p_-}},&q>p_-,\\
0,&q=p_-.
\end{cases}
\end{equation}
To absorb the additive term in the source--potential estimate, we use $C_*:=\int_\Omega c_{p_-}(p(x))\,dx$
and define the \emph{shifted energy}
\begin{equation*}
\E(t)=E(t)+\frac{C_*}{p_-}.
\end{equation*}
The function $c_{p_-}$ is continuous on $[p_-,p_+]$, so
\begin{equation*}
0\le C_*=\int_\Omega c_{p_-}(p(x))\,dx\le |\Omega|\max_{q\in[p_-,p_+]}c_{p_-}(q)<\infty.
\end{equation*}
Since the shift $C_*/p_-$ is independent of time, $\E'=E'$ wherever the derivatives exist.

\begin{theorem}[Local well-posedness]\label{local}
Assume \textup{(H1)--(H2)}. For every $(u_0,u_1)\in V\times H$, problem~\eqref{problem} admits a unique maximal energy solution on $[0,\Tmax)$, where $\Tmax\in(0,\infty]$. The solution depends continuously on the initial data on every compact subinterval of its existence interval. If $\Tmax<\infty$, then
\begin{equation}\label{28}
\limsup_{t\uparrow\Tmax}\left(\norm{u(t)}_V+\norm{u_t(t)}_2\right)=\infty.
\end{equation}
The energy is locally absolutely continuous, and
\begin{equation}\label{ei}
E'(t)=-\norm{\nabla u_t(t)}_2^2-\frac12g(t)\norm{\nabla u(t)}_2^2
+\frac12\int_0^t g'(t-s)\norm{\nabla u(t)-\nabla u(s)}_2^2\,ds
\end{equation}
for almost every $t<\Tmax$.
\end{theorem}
The proof of Theorem~\ref{local} proceeds directly via Faedo-Galerkin method combined with the Banach fixed-point theorem. We omit the detailed argument here. In particular, for $0\le s\le t<\Tmax$, integrating \eqref{ei} over $[s,t]$ and using $g\ge0$ and $g'\le0$ shows
\begin{equation}\label{dissipation}
\E(t)+\int_s^t\norm{\nabla u_r(r)}_2^2\,dr\le\E(s).
\end{equation}
Here we used that the shift is constant.

For negative shifted energy, the additional kernel condition is
\begin{equation}\label{kw}
\ell\ge\frac1{(p_--1)^2}.
\end{equation}
For nonnegative shifted energy, we require the strict condition
\begin{equation}\label{ks}
\ell>\frac1{(p_--1)^2},
\end{equation}
and define
\begin{equation}\label{29}
C=\min\left\{p_-+2,\frac{2\lambda_1[(p_--1)^2\ell-1]}{2p_-+\lambda_1}\right\}>0.
\end{equation}
For example, $g(t)=a e^{-bt}$ with $a\ge0$ and $b>0$ satisfies \eqref{ks} when $a/b<1-(p_--1)^{-2}$.

\begin{theorem}[Blow-up and upper lifespan bounds]\label{blowup}
Assume \textup{(H1)--(H2)}, and let $u$ be the solution in Theorem~\ref{local}.
\begin{enumerate}
\item \textbf{Negative shifted initial energy.}
Suppose \eqref{kw} holds and $\E(0)<0$. Then $\Tmax<\infty$ and \eqref{28} holds. Put $\beta:=-\E(0)>0$. For every $\sigma>0$ such that
\begin{equation*}
(p_--2)\bigl[(u_0,u_1)+\beta\sigma\bigr]>2\norm{\nabla u_0}_2^2,
\end{equation*}
one has
\begin{equation}\label{30}
\Tmax\le
\frac{2(\norm{u_0}_2^2+\beta\sigma^2)}
{(p_--2)[(u_0,u_1)+\beta\sigma]-2\norm{\nabla u_0}_2^2}.
\end{equation}
In particular, minimizing the right-hand side of \eqref{30} over all admissible $\sigma$ gives
\begin{equation}\label{31}
\Tmax\le
\frac{4\left[\sqrt{d_0^2+(p_--2)^2\beta\norm{u_0}_2^2}-d_0\right]}
{(p_--2)^2\beta},
\end{equation}
where 
\begin{equation*}
d_0:=(p_--2)(u_0,u_1)-2\norm{\nabla u_0}_2^2.
\end{equation*}

\item \textbf{Zero shifted initial energy.}
Suppose \eqref{ks} holds, $\E(0)=0$, and $(u_0,u_1)>0$. Then $\Tmax<\infty$ and \eqref{28} holds. Its upper bound is given below by \eqref{35}, with $L_0:=(u_0,u_1)$.

\item \textbf{Positive shifted initial energy.}
Suppose \eqref{ks} holds and
\begin{equation}\label{he}
0<\E(0)<\frac C{p_-}(u_0,u_1).
\end{equation}
Then $\Tmax<\infty$ and \eqref{28} holds. The upper bound is again given by \eqref{35},
without any additional restriction on the initial energy.
\end{enumerate}
For the bounds in \textup{(ii)} and \textup{(iii)}, define, in both cases,
\begin{equation}\label{32}
L_0:=(u_0,u_1)-\frac{p_-}{C}\E(0)>0,\quad
\beta:=\frac{CL_0}{2p_-}.
\end{equation}
Let $\tau$ be the unique positive solution of
\begin{equation}\label{33}
CL_0\tau^2+
\left[L_0-\frac{2(\E(0)+\beta)}{\lambda_1}\right]\tau
-\norm{u_0}_2^2=0.
\end{equation}
Set
\begin{equation}\label{34}
\begin{split}
U_\tau:=\left(\norm{u_0}_2+
\sqrt{\frac{\tau(\E(0)+\beta)}{\lambda_1}}\right)^2,
\quad G_\tau:=\left(\norm{\nabla u_0}_2+
\sqrt{\tau(\E(0)+\beta)}\right)^2,
\end{split}
\end{equation}
and \[d_\tau:=\frac{p_--2}{2}L_0-2G_\tau.\]
Then
\begin{equation}\label{35}
\Tmax\le\tau+
\frac{4\left[\sqrt{d_\tau^2+(p_--2)^2\beta U_\tau}-d_\tau\right]}
{(p_--2)^2\beta}.
\end{equation}
\end{theorem}





\begin{theorem}[Blow-up data at every prescribed energy level]
\label{prescribed}
Assume \textup{(H1)--(H2)} and \eqref{ks}.
For every $R\in\mathbb R$, there exist initial data
$(u_0,u_1)\in C_c^\infty(\Omega)\times C_c^\infty(\Omega)$
such that $E(0)=R$ and the corresponding solution
blows up in finite time.
Moreover, these data satisfy the criterion in
Theorem~\ref{blowup} corresponding to the sign of
$\E(0)=R+C_*/p_-$, and the associated upper lifespan
bound applies.
\end{theorem}

This theorem asserts the existence of suitable blow-up data at each energy level, not blow-up for every datum at that level. Proposition~\ref{comparison} below gives a different construction that demonstrates the removal of \eqref{3} in the constant-exponent case.

\section{The corrected source-potential inequality and lifespan estimates}\label{estimates}
In this section, we develop the estimates responsible for blow-up and the explicit lifespan bounds. The main steps are the full-history estimate in Lemma~\ref{negativelemma} and the passage to a negative shifted energy level in Lemma~\ref{stoppinglemma}.

\begin{lemma}[Optimal pointwise correction]\label{correctionlemma}
For every $q\ge p_->2$ and $r\in\mathbb R$,
\begin{equation}\label{36}
|r|^q\left(
\frac{q-p_-}{q}\log|r|+\frac{p_-}{q^2}
\right)\ge -c_{p_-}(q),
\end{equation}
where the expression at $r=0$ is defined by continuity.
For each fixed $q$, $c_{p_-}(q)$ is the smallest
nonnegative constant for which
\eqref{36} holds for all $r\in\mathbb R$.
Consequently, every energy solution satisfies
\begin{equation}\label{37}
\begin{aligned}
\int_\Omega |u|^{p(x)}\log|u|\,dx
\ge{}&
\frac{p_-}{2}\norm{u_t}_2^2
+\frac{p_-}{2}
\left(1-\int_0^t g(s)\,ds\right)
\norm{\nabla u}_2^2\\
&+\frac{p_-}{2}(g\circ\nabla u)(t)
-p_-\E(t).
\end{aligned}
\end{equation}
\end{lemma}

\begin{proof}
For $q=p_-$, the left-hand side of
\eqref{36} equals
$|r|^{p_-}/p_-\ge0$.
Thus the optimal nonnegative correction is
$c_{p_-}(p_-)=0$.
For  $q>p_-$ and $r\ne0$.
Setting $y=\log|r|$, we obtain
\[
\frac{d}{dy}
\left[
e^{qy}\left(
\frac{q-p_-}{q}y+\frac{p_-}{q^2}
\right)
\right]
=e^{qy}\bigl[(q-p_-)y+1\bigr].
\]
The derivative is negative for $y<-1/(q-p_-)$
and positive for $y>-1/(q-p_-)$.
Hence the global minimum is attained at
$y=-1/(q-p_-)$ and equals
\[
-\frac{q-p_-}{q^2}
e^{-\frac{q}{q-p_-}}
=-c_{p_-}(q).
\]
The expression in \eqref{36} tends
to zero as $r\to0$, so the inequality also holds
at $r=0$.
Moreover, equality holds when
$|r|=e^{-1/(q-p_-)}$, which proves optimality.

By the definition of $F$, taking $q=p(x)$
in \eqref{36} and integrating over
$\Omega$ gives
\[
\int_\Omega |u|^{p(x)}\log|u|\,dx
\ge p_-\int_\Omega F(x,u)\,dx-C_*.
\]
Using \eqref{energy} and
$p_-\E(t)=p_-E(t)+C_*$ yields
\eqref{37}.
\end{proof}

\begin{corollary}\label{correctionsize}
Let $\Delta=p_+-p_-$. If $\Delta>0$, then
\begin{equation}\label{38}
0\le C_*\le\frac{|\Omega|\Delta}{p_-^2}
e^{-1-\frac{p_-}{\Delta}}.
\end{equation}
If $\Delta=0$, then $C_*=0$.
\end{corollary}
\begin{proof}
Set $q=p_-+d$, with $0<d\le\Delta$. Then
\[
c_{p_-}(q)=\frac{d}{(p_-+d)^2}e^{-1-p_-/d}
\le\frac{\Delta}{p_-^2}e^{-1-p_-/\Delta}.
\]
The same upper bound holds at $q=p_-$. Integration gives \eqref{38}. The case $\Delta=0$ follows directly from \eqref{correction}.
\end{proof}

\begin{remark}[Constant exponents]\label{constantcase}
If $p(x)\equiv p$, then $C_*=0$ and $\E=E$. Condition~\eqref{he} becomes \eqref{2}, with the same constant $C$ as in~\cite{Liao}. The present lifespan estimate uses a different argument and does not require \eqref{3}. More generally, \eqref{38} shows that $C_*\to0$ as $p_+-p_-\to0$ when $p_-$ stays bounded below by a fixed number greater than $2$.
\end{remark}

The following growth estimate is an adaptation of
\cite[Lemma~2.7]{Liao} to the shifted energy $\E$.
The key ingredient is the corrected source--potential
inequality \eqref{37}.
\begin{lemma}\label{growthlemma}
Assume \textup{(H1)--(H2)} and \eqref{ks}.
Define
\begin{equation*}
L(t)=(u(t),u_t(t))-\frac{p_-}{C}\E(t).
\end{equation*}
Then $L$ is locally absolutely continuous on
$[0,T_{\max})$ and satisfies
\begin{equation*}
L(t)\ge L(0)e^{Ct}
\quad\text{for }0\le t<T_{\max}.
\end{equation*}
\end{lemma}

\begin{proof}
Fix $T<T_{\max}$.
The regularity \eqref{regularity} and the energy
identity imply that $L$ is absolutely continuous
on $[0,T]$. Using $\eqref{problem}_1$ gives
\begin{equation}\label{39}
\begin{aligned}
\frac{d}{dt}(u,u_t)
={}&\norm{u_t}_2^2
-\left(1-\int_0^t g(s)\,ds\right)
 \norm{\nabla u}_2^2
-(\nabla u,\nabla u_t)\\
&+\int_0^t g(t-s)
  (\nabla u(t),\nabla u(s)-\nabla u(t))\,ds
+\int_\Omega |u|^{p(x)}\log|u|\,dx
\end{aligned}
\end{equation}
for almost every $t\in(0,T)$.
By Cauchy--Schwarz and Young's inequalities, one has
\begin{equation}\label{40}
\begin{aligned}
\int_0^t g(t-s)
(\nabla u(t),\nabla u(s)-\nabla u(t))\,ds
\ge
-\frac{1}{2p_-}
 \left(\int_0^t g(s)\,ds\right)
 \norm{\nabla u(t)}_2^2
-\frac{p_-}{2}(g\circ\nabla u)(t),
\end{aligned}
\end{equation}
\[
-(\nabla u,\nabla u_t)
\ge
-\frac{C}{4p_-}\norm{\nabla u}_2^2
-\frac{p_-}{C}\norm{\nabla u_t}_2^2.
\]
Substituting these estimates and
\eqref{37} into
\eqref{39}, we get

\begin{equation*}
\begin{aligned}
\frac{d}{dt}(u,u_t)
={}&\frac{p_-+2}{2}\norm{u_t}_2^2
+\frac{2[(p_--1)^2\ell-1]-C}{4p_-}
 \norm{\nabla u}_2^2\\
&-\frac{p_-}{C}\norm{\nabla u_t}_2^2+\frac{p_-}{2}\norm{u_t}_2^2
-p_-\E(t).
\end{aligned}
\end{equation*}
Here we used
\[
\begin{aligned}
\frac{p_--2}{2}
 \left(1-\int_0^t g(s)\,ds\right)
-\frac{1}{2p_-}\int_0^t g(s)\,ds
-\frac{C}{4p_-}\ge
\frac{2[(p_--1)^2\ell-1]-C}{4p_-}.
\end{aligned}
\]
Note that
$\E'(t)\le-\norm{\nabla u_t(t)}_2^2$, one gets
\begin{equation}\label{41}
L'(t)\ge
\frac{p_-+2}{2}\norm{u_t}_2^2
+\frac{2[(p_--1)^2\ell-1]-C}{4p_-}
 \norm{\nabla u}_2^2
-p_-\E(t).
\end{equation}
By \eqref{29},
\[
C\le p_-+2,
\quad
\lambda_1
\frac{2[(p_--1)^2\ell-1]-C}{4p_-}
\ge\frac{C}{2}.
\]
Consequently, Poincar\'e's inequality \eqref{poincare} and
$2(u,u_t)\le\norm{u}_2^2+\norm{u_t}_2^2$ yield
\[
\begin{aligned}
L'(t)
\ge \frac{C}{2}
\left(\norm{u_t}_2^2+\norm{u}_2^2\right)
-p_-\E(t)\ge C(u,u_t)-p_-\E(t)
=CL(t).
\end{aligned}
\]
Hence $(e^{-Ct}L(t))'\ge0$ almost everywhere.
Integration gives $L(t)\ge e^{Ct}L(0)$ on $[0,T]$.
Since $T<T_{\max}$ was arbitrary, the proof
is complete.
\end{proof}

We use the classical concavity criterion associated with Levine's method~\cite{Levine}, see also \cite[Lemma~2.6]{Liao} and~\cite{LiaoGao}. The brief proof records the time regularity needed here.
\begin{lemma}\label{concavity}
Let $M\in W^{2,1}(0,R)\cap C^1([0,R])$ be positive and satisfy
\[
MM''-(1+\theta)(M')^2\ge0\quad\text{a.e. on }(0,R)
\]
for some $\theta>0$. If $M'(0)>0$, such a function cannot exist on $[0,R]$ when
\begin{equation*}
R>\frac{M(0)}{\theta M'(0)}.
\end{equation*}
\end{lemma}
\begin{proof}
Since $M\in C^1([0,R])$ and $M>0$ on the compact interval $[0,R]$, there exists
$m_0>0$ such that
\[
M(s)\ge m_0\quad\text{for all }s\in[0,R].
\]
 The Sobolev chain rule gives
\[
(M^{-\theta})''=-\theta M^{-\theta-2}
\bigl[MM''-(1+\theta)(M')^2\bigr]\le0,
\]
which reveals that $M^{-\theta}$ is concave on $[0,R]$.
Hence $M(s)^{-\theta}\le M(0)^{-\theta}-\theta M(0)^{-\theta-1}M'(0)s$. Equivalently,
\[
M(s)^{-\theta}
\le
M(0)^{-\theta}
\left(
1-\frac{\theta M'(0)}{M(0)}s
\right).
\]
If
$
s>\frac{M(0)}{\theta M'(0)},
$
then the right-hand side is strictly negative. This contradicts the positivity of the left-hand side, since $M(s)>0$ implies
$
M(s)^{-\theta}>0.
$
Therefore no point $s$ in the interval $[0,R]$ can exceed
$M(0)/(\theta M'(0))$. This completes the proof.
\end{proof}
The following argument starts at an arbitrary time $t_0$ with negative shifted energy. The dissipation inequality is applied on $[t_0,t]$, while the memory integral retains its lower limit zero. 
\begin{lemma}[Negative shifted energy with the full history]\label{negativelemma}
Assume \textup{(H1)--(H2)} and \eqref{kw}. If $\E(t_0)<0$ for some $t_0<\Tmax$, put $\beta=-\E(t_0)>0$. For every $\sigma>0$ satisfying
\begin{equation}\label{42}
(p_--2)\bigl[(u(t_0),u_t(t_0))+\beta\sigma\bigr]>2\norm{\nabla u(t_0)}_2^2,
\end{equation}
one has
\begin{equation}\label{43}
\Tmax\le t_0+
\frac{2(\norm{u(t_0)}_2^2+\beta\sigma^2)}
{(p_--2)[(u(t_0),u_t(t_0))+\beta\sigma]-2\norm{\nabla u(t_0)}_2^2}.
\end{equation}
\end{lemma}
\begin{proof}
Fix an admissible $\sigma$. If \eqref{43} fails, let us choose $R$ satisfying
\[\frac{2\bigl(\|u(t_0)\|_2^2+\beta\sigma^2\bigr)}
{(p_--2)\bigl[(u(t_0),u_t(t_0))+\beta\sigma\bigr]
-2\|\nabla u(t_0)\|_2^2}
<R<T_{\max}-t_0.\]
 For $0\le s\le R$, put $t=t_0+s$ and define
\begin{equation*}
M(s)=\norm{u(t)}_2^2+\int_{t_0}^t\norm{\nabla u(r)}_2^2\,dr
+(R-s)\norm{\nabla u(t_0)}_2^2+\beta(s+\sigma)^2.
\end{equation*}
\emph{Step 1: the concavity functional and its regularity.}
The term $\beta(s+\sigma)^2$ makes $M$ strictly positive.
The product rule used in \eqref{39} and $u\in H^1(t_0,t_0+R;V)$ imply $M\in C^1([0,R])\cap W^{2,1}(0,R)$. In particular,
\begin{equation*}
M'(s)=2\left[(u,u_t)+\int_{t_0}^t(\nabla u,\nabla u_r)\,dr+\beta(s+\sigma)\right].
\end{equation*}
 $M'$ is continuous on $[0,R]$ because
$u\in C([t_0,t_0+R];V)$ and
$u_t\in C([t_0,t_0+R];H)$.
Moreover, the weak product rule gives
\begin{equation}
\label{1}
M''(s)
=2\norm{u_t(t)}_2^2
+2\langle u_{tt}(t),u(t)
+2(\nabla u(t),\nabla u_t(t))
+2\beta
\end{equation}
for almost every $s\in(0,R)$.
All terms belong to $L^1(0,R)$.
Indeed, $u$ is bounded in $V$ on
$[t_0,t_0+R]$, while
$u_t\in L^2(t_0,t_0+R;V)$ and
$u_{tt}\in L^2(t_0,t_0+R;V^*)$.
Thus
$M\in C^1([0,R])\cap W^{2,1}(0,R)$,
as required by Lemma~\ref{concavity}.

Cauchy--Schwarz in the three components indicates
\begin{equation}\label{45}
\begin{split}
\frac14(M'(s))^2
&\le\left[\norm{u(t)}_2^2+\int_{t_0}^t\norm{\nabla u}_2^2\,dr+\beta(s+\sigma)^2\right]
\left[\norm{u_t(t)}_2^2+\int_{t_0}^t\norm{\nabla u_r}_2^2\,dr+\beta\right]\\
&\le M(s)\left[\norm{u_t(t)}_2^2+\int_{t_0}^t\norm{\nabla u_r}_2^2\,dr+\beta\right].
\end{split}
\end{equation}
The last inequality uses \((R-s)\|\nabla u(t_0)\|_2^2\ge0.\)

\emph{Step 2: the concavity inequality.}
Let us combine \eqref{1} and \eqref{39}, then
\begin{equation*}
\begin{split}
M''(s)={}&2\norm{u_t}_2^2-2\left(1-\int_0^t g(r)\,dr\right)\norm{\nabla u}_2^2+2\beta\\
&+2\int_0^t g(t-r)(\nabla u(t),\nabla u(r)-\nabla u(t))\,dr
+2\int_\Omega |u|^{p(x)}\log|u|\,dx.
\end{split}
\end{equation*}
 Applying \eqref{37} and \eqref{40} yields
\begin{equation*}
M''(s)\ge(p_-+2)\norm{u_t}_2^2+
\frac{(p_--1)^2\ell-1}{p_-}\norm{\nabla u}_2^2-2p_-\E(t)+2\beta.
\end{equation*}
By \eqref{dissipation}, $\E(t)\le-\beta-\int_{t_0}^t\norm{\nabla u_r}_2^2\,dr$. Consequently,
\begin{equation}\label{48}
\begin{split}
&M''(s)-(p_-+2)\left[\norm{u_t}_2^2+\int_{t_0}^t\norm{\nabla u_r}_2^2\,dr+\beta\right]\\
&\quad\ge\frac{(p_--1)^2\ell-1}{p_-}\norm{\nabla u}_2^2
+(p_--2)\int_{t_0}^t\norm{\nabla u_r}_2^2\,dr+p_-\beta\ge0.
\end{split}
\end{equation}
Combining \eqref{45} and \eqref{48}, we have 
\begin{equation}\label{50}
M(s)M''(s)-\frac{p_-+2}{4}(M'(s))^2\ge0.
\end{equation}

\emph{Step 3: the finite comparison interval.}
Set $\theta=(p_--2)/4$. Condition~\eqref{42} implies $M'(0)>0$. Direct computation shows 
\[
\frac{M(0)}{\theta M'(0)}=
\frac{2[\norm{u(t_0)}_2^2+R\norm{\nabla u(t_0)}_2^2+\beta\sigma^2]}
{(p_--2)[(u(t_0),u_t(t_0))+\beta\sigma]}<R.
\]
The last inequality is exactly the choice of $R$ above. Since $t_0+R<\Tmax$, the solution supplies a finite positive $M$ on $[0,R]$, contradicting Lemma~\ref{concavity}. This proves \eqref{43}.
\end{proof}

\begin{corollary}\label{optimization}
Under the assumptions of Lemma~\ref{negativelemma}, put
\[
d_{t_0}=(p_--2)(u(t_0),u_t(t_0))-2\norm{\nabla u(t_0)}_2^2.
\]
The fraction in \eqref{43} is minimized at
\begin{equation}\label{51}
\sigma_*=
\frac{\sqrt{d_{t_0}^2+(p_--2)^2\beta\norm{u(t_0)}_2^2}-d_{t_0}}{(p_--2)\beta},
\end{equation}
and
\begin{equation}\label{52}
\Tmax\le t_0+
\frac{4[\sqrt{d_{t_0}^2+(p_--2)^2\beta\norm{u(t_0)}_2^2}-d_{t_0}]}{(p_--2)^2\beta}.
\end{equation}
\end{corollary}
\begin{proof}
If $u(t_0)=0$, the elastic and potential terms vanish and the shifted energy
\[\mathcal E(t_0)
=\frac12\|u_t(t_0)\|_2^2
+\frac12(g\circ\nabla u)(t_0)
+\frac{C_*}{p_-}\ge0,\]
contradicting $\E(t_0)<0$. Hence $U_{t_0}:=\norm{u(t_0)}_2^2>0$. For admissible $\sigma$, set
\[
B(\sigma):=\frac{2(U_{t_0}+\beta\sigma^2)}{d_{t_0}+(p_--2)\beta\sigma},\quad
B'(\sigma)=\frac{2\beta[(p_--2)\beta\sigma^2+2d_{t_0}\sigma-(p_--2)U_{t_0}]}
{[d_{t_0}+(p_--2)\beta\sigma]^2}.
\]
Obviously, $(p_--2)\beta\sigma^2+2d_{t_0}\sigma-(p_--2)U_{t_0}=0$ has exactly one positive root, namely \eqref{51}. At that root,
\[d_{t_0}+(p_--2)\beta\sigma_* =\sqrt{d_{t_0}^2+(p_--2)^2\beta U_{t_0}}>|d_{t_0}|>0,\]
 so it is admissible. On the admissible interval
$\{\sigma>0:d_{t_0}+(p_--2)\beta\sigma>0\}$, we find
 \[B'(\sigma)<0\quad\text{for }0<\sigma<\sigma_*,
\quad
B'(\sigma)>0\quad\text{for }\sigma>\sigma_*.\]
Thus the root is the unique minimizer. Using \eqref{51}, then 
\[
U_{t_0}+\beta\sigma_*^2
=\frac{2\sigma_*}{p_--2}[d_{t_0}+(p_--2)\beta\sigma_*],\]
further,
\[B(\sigma_*)=\frac{4\sigma_*}{p_--2}=
\frac{4[\sqrt{d_{t_0}^2+(p_--2)^2\beta\norm{u(t_0)}_2^2}-d_{t_0}]}{(p_--2)^2\beta},
\]
which verifies \eqref{52}.
\end{proof}

The next estimate is the quantitative step needed when the shifted initial energy is nonnegative. The bounds on $[0,t_\beta]$ follow from the dissipation inequality and do not require monotonicity of $\|u(t)\|_2$.
\begin{lemma}\label{stoppinglemma}
Assume \textup{(H1)--(H2)} and \eqref{ks}. Suppose $\E(0)\ge0$ and $L_0>0$, with $L_0$ and $\beta$ as in \eqref{32}. Let $\tau$, $U_\tau$, and $G_\tau$ be given by \eqref{33} and \eqref{34}. Then either $\Tmax\le\tau$, or there is $t_\beta\in(0,\tau)$ such that
\begin{equation}\label{53}
\begin{gathered}
\E(t_\beta)=-\beta,\quad (u(t_\beta),u_t(t_\beta))\ge L_0/2,\\
\norm{u(t_\beta)}_2^2\le U_\tau,\quad
\norm{\nabla u(t_\beta)}_2^2\le G_\tau.
\end{gathered}
\end{equation}
\end{lemma}
\begin{proof}
Suppose $\Tmax>\tau$. Since $\E(0)\ge0$ and $L_0>0$, we have $(u_0,u_1)
=L_0+\frac{p_-}{C}\mathcal E(0)>0$, and hence $u_0\ne0$. Obviously, \eqref{33}  has exactly one positive root.

\emph{Step 1: }
For any $t\le\tau$ for which $\E(s)\ge-\beta$ on $[0,t]$, \eqref{dissipation} directly illustrates 
\begin{equation*}
\int_0^t\norm{\nabla u_s}_2^2\,ds\le\E(0)-\E(t)\le\E(0)+\beta.
\end{equation*}
The representation $u(t)=u_0+\int_0^t u_s\,ds$, Poincar\'e's inequality \eqref{poincare}, and Cauchy--Schwarz imply
\begin{equation}\label{54}
\begin{split}
\norm{u(t)}_2&\le\|u_0\|_2+\int_0^t\|u_s(s)\|_2\,ds\le\|u_0\|_2+
\sqrt{\frac t{\lambda_1}}
\left(\int_0^t\|\nabla u_s(s)\|_2^2\,ds\right)^{1/2}\\
&\le\|u_0\|_2+
\sqrt{\frac{t(\mathcal E(0)+\beta)}{\lambda_1}}
,\\
\norm{\nabla u(t)}_2&\le\|\nabla u_0\|_2+
\int_0^t\|\nabla u_s(s)\|_2\,ds
\le\norm{\nabla u_0}_2+
\sqrt{t(\E(0)+\beta)}.
\end{split}
\end{equation}

\emph{Step 2: }
Suppose, for a contradiction, that $\E(s)\ge-\beta$ for all $s\in[0,\tau]$. Lemma~\ref{growthlemma} gives
\[
\frac d{dt}\norm{u(t)}_2^2
=2L(t)+\frac{2p_-}{C}\E(t)
\ge2L_0e^{Ct}-\frac{2p_-\beta}{C}=2L_0e^{Ct}-L_0.
\]
Therefore
\begin{equation*}
\norm{u(t)}_2^2\ge\norm{u_0}_2^2+\frac{2L_0}{C}(e^{Ct}-1)-L_0t.
\end{equation*}
Using $e^{Ct}-1>Ct+(Ct)^2/2$ for $t>0$, then
\[
\norm{u(\tau)}_2^2>
\norm{u_0}_2^2+L_0\tau+CL_0\tau^2
=2\norm{u_0}_2^2+\frac{2\tau(\E(0)+\beta)}{\lambda_1},
\]
where the equality follows from \eqref{33}. In contrast, squaring the first inequality in \eqref{54} and using $(a+b)^2\le2a^2+2b^2$ yields
\[
\norm{u(\tau)}_2^2\le
2\norm{u_0}_2^2+\frac{2\tau(\E(0)+\beta)}{\lambda_1}.
\] 
This is a contradiction.

\emph{Step 3: }
The contradiction in Step~2 shows that
$\mathcal E(s_0)<-\beta$ for some $s_0\in(0,\tau]$.
If $s_0=\tau$, continuity gives the same strict inequality
at a time slightly earlier than $\tau$.
Thus there exists $s_1\in(0,\tau)$ such that
$\mathcal E(s_1)<-\beta$.
Since $\mathcal E(0)\ge0\ge-\beta$, the intermediate value theorem
provides $t_\beta\in(0,s_1)$ with
$\mathcal E(t_\beta)=-\beta$.

The energy is nonincreasing, so
$\mathcal E(t)\ge\mathcal E(t_\beta)=-\beta$
for every $t\in[0,t_\beta]$.
Hence \eqref{54} applies at $t_\beta$.
Using $t_\beta<\tau$ and \eqref{34}, we obtain
\[
\begin{aligned}
\|u(t_\beta)\|_2^2
&\le
\left(
\|u_0\|_2+
\sqrt{\frac{t_\beta(\mathcal E(0)+\beta)}{\lambda_1}}
\right)^2
\le U_\tau,\\
\|\nabla u(t_\beta)\|_2^2
&\le
\left(
\|\nabla u_0\|_2+
\sqrt{t_\beta(\mathcal E(0)+\beta)}
\right)^2
\le G_\tau.
\end{aligned}
\]
Moreover, the definition of $L$, Lemma~\ref{growthlemma},
and \eqref{32} establish 
\[
\begin{aligned}
(u(t_\beta),u_t(t_\beta))=L(t_\beta)+\frac{p_-}{C}\mathcal E(t_\beta)
=L(t_\beta)-\frac{p_-\beta}{C}\ge L_0e^{Ct_\beta}-\frac{L_0}{2}
\ge\frac{L_0}{2}.
\end{aligned}
\]
Together with $\mathcal E(t_\beta)=-\beta$ and
$0<t_\beta<\tau$, these estimates prove \eqref{53}.\end{proof}

To show that the blow-up criterion can be satisfied at every prescribed initial energy level, we need to construct suitable initial data. We choose the initial displacement in the form $u_0=s\varphi$ and adjust the initial velocity to attain the prescribed energy. This construction requires the source potential along $s\varphi$ to grow faster than the quadratic elastic energy. The following lemma establishes this property and also allows the initial displacement--velocity pairing to become arbitrarily large while the initial energy remains fixed.
\begin{lemma}\label{raylemma}
There exists $\varphi\in C_c^\infty(\Omega)$, $0\le\varphi\le1$, such that
\begin{equation}\label{55}
\frac1{s^2}\int_\Omega F(x,s\varphi)\,dx\longrightarrow\infty
\quad\text{as }s\to\infty.
\end{equation}
\end{lemma}
\begin{proof}
Choose $\varphi=1$ on a ball $A$ compactly contained in $\Omega$. For fixed $q\ge p_-$, the potential $|r|^q\log|r|/q-|r|^q/q^2$ is minimized at $|r|=1$, with value $-1/q^2$. Thus $F(x,r)\ge-1/p_-^2$. For $s>e^{1/p_-}$ and $x\in A$,
\[
F(x,s)=\frac{s^{p(x)}}{p(x)}\left(\log s-\frac1{p(x)}\right)
\ge\frac{s^{p_-}}{p_+}\left(\log s-\frac1{p_-}\right).
\]
Consequently,
\[
\int_\Omega F(x,s\varphi)\,dx=\int_A F(x,s)\,dx
+\int_{\Omega\setminus A}F(x,s\varphi)\,dx
\ge\frac{|A|s^{p_-}}{p_+}\left(\log s-\frac1{p_-}\right)-\frac{|\Omega|}{p_-^2}.
\]
Division by $s^2$ and $p_->2$ prove \eqref{55}.
\end{proof}

\section{Blow-up bounds and constructions of initial data}\label{proofs}

\begin{proof}[\textbf{Proof of Theorem~\ref{blowup}}]
In case \textup{(i)}, Lemma~\ref{negativelemma} at $t_0=0$ gives \eqref{30}, and Corollary~\ref{optimization} gives \eqref{31}.

For cases \textup{(ii)} and \textup{(iii)}, one has $\E(0)\ge0$ and $L_0>0$. If $\Tmax\le\tau$, \eqref{35} holds because its added term is positive. Otherwise Lemma~\ref{stoppinglemma} supplies $t_\beta<\tau$ satisfying \eqref{53}.
For every $\sigma>0$ with $d_\tau+(p_--2)\beta\sigma>0$, these estimates reveal
\begin{equation}\label{aa}
\begin{split}
&(p_--2)\bigl[(u(t_\beta),u_t(t_\beta))+\beta\sigma\bigr]
-2\norm{\nabla u(t_\beta)}_2^2\\
&\quad\ge\frac{p_--2}{2}L_0-2G_\tau+(p_--2)\beta\sigma
=d_\tau+(p_--2)\beta\sigma>0.
\end{split}
\end{equation}
Thus $\sigma$ is admissible in Lemma~\ref{negativelemma}. Applying Lemma~\ref{negativelemma} at $t_0=t_\beta$,
using $t_\beta<\tau$ and
$\norm{u(t_\beta)}_2^2\le U_\tau$, and
in \eqref{aa}, we obtain
\begin{equation}\label{aa1}
\begin{aligned}
\Tmax
&\le t_\beta+
\frac{2\bigl(\norm{u(t_\beta)}_2^2+\beta\sigma^2\bigr)}
{(p_--2)\bigl[(u(t_\beta),u_t(t_\beta))+\beta\sigma\bigr]
-2\norm{\nabla u(t_\beta)}_2^2}\\
&\le\tau+
\frac{2(U_\tau+\beta\sigma^2)}
{d_\tau+(p_--2)\beta\sigma}.
\end{aligned}
\end{equation}
Applying
Corollary~\ref{optimization} with $U_{t_0}$ replaced by $U_\tau>0$
and $d_{t_0}$ replaced by $d_\tau$, we minimize the right-hand side
of \eqref{aa1} by taking
\[
\sigma=
\frac{\sqrt{d_\tau^2+(p_--2)^2\beta U_\tau}-d_\tau}
{(p_--2)\beta}>0.
\]
This choice is admissible because
\[
d_\tau+(p_--2)\beta\sigma
=\sqrt{d_\tau^2+(p_--2)^2\beta U_\tau}>0.
\]
Substituting this value into \eqref{aa1}
gives \eqref{35}.
Since the resulting upper bound is finite,
we conclude that $\Tmax<\infty$.
In all three cases, the 28 criterion
in Theorem~\ref{local} yields \eqref{28}.

\end{proof}

\begin{proof}[\textbf{Proof of Theorem~\ref{prescribed}}]
Let $\varphi$ be as in Lemma~\ref{raylemma} and fix $R\in\R$. For large $s$, define
\begin{equation}\label{aa2}
b_s=\left[\frac{2}{\norm\varphi_2^2}
\left(R+\int_\Omega F(x,s\varphi)\,dx-\frac{s^2}{2}\norm{\nabla\varphi}_2^2\right)\right]^{1/2},
\quad u_0=s\varphi,\quad u_1=b_s\varphi.
\end{equation}
By \eqref{55}, one has
\[\frac{2}{\norm\varphi_2^2}
\left(R+\int_\Omega F(x,s\varphi)\,dx-\frac{s^2}{2}\norm{\nabla\varphi}_2^2\right)>0\]
For $t=0$, we have
\[
E(0)=\frac{b_s^2}{2}\norm\varphi_2^2+
\frac{s^2}{2}\norm{\nabla\varphi}_2^2-\int_\Omega F(x,s\varphi)\,dx=R.
\]
Moreover,
\[
\frac{b_s^2}{s^2}=\frac{2}{\norm\varphi_2^2}
\left[\frac R{s^2}+\frac1{s^2}\int_\Omega F(x,s\varphi)\,dx-
\frac12\norm{\nabla\varphi}_2^2\right]\longrightarrow\infty\quad\text{as }s\to\infty.
\]
Thus $b_s/s\to\infty$ and $(u_0,u_1)=s b_s\norm\varphi_2^2\to\infty$  as $s\to \infty$, whereas $\E(0)=R+C_*/p_-$ is fixed. If $\mathcal E(0)<0$, Theorem~\ref{blowup}\textup{(i)}
applies without any restriction on $(u_0,u_1)$. If $\mathcal E(0)=0$, then
$(u_0,u_1)=sb_s\|\varphi\|_2^2>0$,
so Theorem~\ref{blowup}\textup{(ii)} applies. If $\mathcal E(0)>0$, then $(u_0,u_1)\to\infty$ as
$s\to\infty$, whereas $\mathcal E(0)$ remains fixed.
Hence, for sufficiently large $s$,
$
0<\mathcal E(0)<\frac{C}{p_-}(u_0,u_1),
$
and Theorem~\ref{blowup}\textup{(iii)} applies.
\end{proof}

The preceding construction fixes the energy. The next one lets it increase and demonstrates that removing \eqref{3} enlarges the class covered by the upper lifespan estimate.
\begin{proposition}[Data outside the additional lifespan restriction]\label{comparison}
Let $p(x)\equiv p$ with $2<p<2^*$, and assume \textup{(H2)} and \eqref{ks}. There are smooth compactly supported data with arbitrarily large initial energy satisfying
\begin{equation}\label{aa3}
\frac{C}{2p}\norm{u_0}_2^2<E(0)<\frac Cp(u_0,u_1).
\end{equation}
The solutions satisfy the upper bound \eqref{35}. In the common exponent range of the two papers, these data violate the additional hypothesis in \cite[Theorem~3.2]{Liao} while satisfying its high-energy criterion.
\end{proposition}
\begin{proof}

Let $\varphi$ be as in Lemma~\ref{raylemma} and choose
$\kappa>C\norm\varphi_2^2/(2p)$. In \eqref{aa2}, replace the fixed $R$ by $R(s)=\kappa s^2$. For large $s$,
\[\frac{2}{\norm\varphi_2^2}
\left(\kappa s^2+\int_\Omega F(x,s\varphi)\,dx-\frac{s^2}{2}\norm{\nabla\varphi}_2^2\right)>0.\]
The same calculation as in the proof of Theorem~\ref{prescribed} yields $E(0)=\kappa s^2\to\infty$ as $s\to \infty$. Also,
\[
\frac{b_s^2}{s^2}=\frac{2}{\norm\varphi_2^2}
\left[\kappa+\frac1{s^2}\int_\Omega F(s\varphi)\,dx-
\frac12\norm{\nabla\varphi}_2^2\right]\longrightarrow\infty\quad\text{as }s\to\infty.
\]
Consequently,
\[
\frac{E(0)}{(u_0,u_1)}=
\frac{\kappa s}{b_s\norm\varphi_2^2}\longrightarrow0\quad\text{as }s\to\infty,
\]
which proves the right inequality in \eqref{aa3} for large $s$. The choice of $\kappa$ yields the left inequality for every such $s$. Since $C_*=0$, Theorem~\ref{blowup}\textup{(iii)} applies.
\end{proof}


\section{Numerical examples}\label{numerical}
In this section, we consider three examples corresponding to the
shifted initial energies $-1$, $0$, and $1$.
The finite-time blow-up follows from
Theorem~\ref{blowup}.

Let us choose
\[
\Omega=(0,\pi),\quad
p(x)=3+\sin^2x,\quad
g(t)=\frac15e^{-t},
\]
with $u(0,t)=u(\pi,t)=0$. Then
\[
p_-=3,\quad p_+=4,\quad
\lambda_1=1,\quad
\ell=1-\int_0^\infty g(s)\,ds=\frac45.
\]
In particular, we have
\[
\ell=\frac45>\frac1{(p_--1)^2}=\frac14,
\quad
C=\min\left\{
5,\frac{2[4(4/5)-1]}{7}
\right\}=\frac{22}{35}.
\]
Thus the strict kernel condition holds,
and \textup{(H1)} is obviously satisfied for $n=1$.

Choose
\[
u_0(x)=3\sin x,\quad u_1(x)=B\sin x,
\]
where $B>0$ is determined by the desired initial
shifted energy.  It is clear that $(u_0(x),u_1(x))\in H_0^1(0,\pi)\times L^2(0,\pi)$ satisfies the boundary conditions.
Let us write
$P=\int_0^\pi F(x,3\sin x)\,dx$, then
\[
\mathcal E(0)
=\frac12\|u_1\|_2^2
+\frac12\|\partial_xu_0\|_2^2
-\int_0^\pi F(x,u_0)\,dx+\frac{C_*}{3}
=\frac{\pi B^2}{4}+\frac{9\pi}{4}
-P+\frac{C_*}{3}.
\]
For each $\delta\in\{-1,0,1\}$, set
\begin{equation}\label{n1}
B=\left[
\frac4\pi\left(\delta-\frac{C_*}{3}+P-\frac{9\pi}{4}\right)
\right]^{1/2},
\end{equation}
which gives $\mathcal E(0)=\delta$.

Numerical quadrature shows
\[
\begin{aligned}
P
&=\int_0^\pi
(3\sin x)^{p(x)}
\left[
\frac{\log(3\sin x)}{p(x)}
-\frac{1}{p(x)^2}
\right]\,dx
\approx16.0087328718,\\
C_*
&=\int_0^\pi
\frac{\sin^2x}{p(x)^2}
\exp\left(-\frac{p(x)}{\sin^2x}\right)\,dx
\approx0.000973763791,
\end{aligned}
\]where the integrands are extended continuously
by zero at the endpoints. The initial data satisfy the blow-up criteria
independently of these quadrature approximations.
Indeed, $F(x,r)\ge-1/9$ for all $r$, while
\[
F(x,r)\ge\frac{r^{p(x)}}4
\left(\log r-\frac13\right)
\quad\text{for }r>e^{1/3}.
\]
Applying this estimate on
$[\pi/3,5\pi/12]$, $[5\pi/12,\pi/2]$,
and their reflections about $\pi/2$ indicates
\[
\begin{aligned}
P\ge{}&\frac{\pi}{24}
\left\{
\left(3\sin\frac\pi3\right)^{15/4}
\left[\log\left(3\sin\frac\pi3\right)-\frac13\right]
\right.\\
&\left.\quad+
\left(3\sin\frac{5\pi}{12}\right)^{(14+\sqrt3)/4}
\left[\log\left(3\sin\frac{5\pi}{12}\right)-\frac13\right]
\right\}
-\frac{2\pi}{27}>8.96.
\end{aligned}
\]
Moreover, using Corollary~\ref{correctionsize} yields
$C_*/3\le\pi e^{-4}/27<0.0022$. Therefore, the positivity of the radicand in \eqref{n1} is fulfilled for all three choices of $\delta$.
The negative shifted-energy criterion applies when $\delta=-1$,
and the zero shifted-energy criterion applies when $\delta=0$
because $(u_0,u_1)=3\pi B/2>0$.
For $\delta=1$, the bounds $P>8.96$ and $C_*/3<0.0022$
in \eqref{n1} yield $B>1.9$, thus
\[
0<\mathcal E(0)=1
<\frac C3(u_0,u_1)=\frac{11\pi B}{35}.
\]
By Theorem~\ref{blowup},  we verify that all three continuous solutions blow up in
finite time.

Using
\[
\|u_0\|_2^2=\|\partial_xu_0\|_2^2=\frac{9\pi}{2},
\quad (u_0,u_1)=\frac{3\pi B}{2},
\]
we evaluate \eqref{31} for $\delta=-1$
and \eqref{35} for $\delta=0,1$.
Table~\ref{numericaldata} lists the initial parameters
and the resulting theoretical upper lifespan bounds.

\begin{table}[htbp]
\centering
\caption{Initial parameters and numerical evaluations
of the theoretical upper lifespan bounds.}
\label{numericaldata}
\label{numericalupperbounds}
\begin{tabular}{c c c c c}
\hline
Case & $\mathcal E(0)$ & $B$ & $(u_0,u_1)$
& Upper lifespan bound\\
\hline
\textup{(i)}
& $-1$ & $3.17951237$ & $14.98309907$ & $108.41623$\\
\textup{(ii)}
& $0$ & $3.37380178$ & $15.89866635$ & $189.92038$\\
\textup{(iii)}
& $1$ & $3.55749603$ & $16.76430510$ & $312.54329$\\
\hline
\end{tabular}
\end{table}

\vspace{1em}
\noindent\textbf{Declaration of competing interest}

The authors declare that they have no known competing financial interests or personal relationships that could have appeared to influence the work reported in this paper.

\vspace{1em}
\noindent\textbf{Data availability}

No data was used for the research described in the article.

\end{document}